\documentclass[11pt,reqno]{amsart}

\usepackage[T1]{fontenc}
\usepackage[utf8]{inputenc}
\usepackage{lmodern}
\usepackage{microtype}
\usepackage{amsmath,amssymb,amsthm,mathtools}
\usepackage{enumitem}
\usepackage[colorlinks=true,citecolor=red,linkcolor=red,urlcolor=blue, backref]{hyperref}
\hypersetup{
  pdftitle={Rigid Functions, IP-Systems, and Topological Mild Mixing},
  pdfauthor={Song Shao, Hui Xu},
  pdfkeywords={IP-system, rigid function, uniform rigidity, mild mixing,
  IP-star, SIP-star, uniformly rigid factor}
}
\allowdisplaybreaks[2]
\numberwithin{equation}{section}
\setlist{itemsep=0.25em,topsep=0.5em}

\theoremstyle{plain}
\newtheorem{theorem}{Theorem}[section]
\newtheorem{thma}{Theorem}

\newtheorem{proposition}[theorem]{Proposition}
\newtheorem{lemma}[theorem]{Lemma}
\newtheorem{corollary}[theorem]{Corollary}
\theoremstyle{definition}
\newtheorem{definition}[theorem]{Definition}
\newtheorem{question}[theorem]{Question}
\theoremstyle{remark}
\newtheorem{remark}[theorem]{Remark}
\newtheorem{example}[theorem]{Example}

\newcommand{\N}{\mathbb N}
\newcommand{\Nzero}{\mathbb N_0}
\newcommand{\Z}{\mathbb Z}
\newcommand{\Torus}{\mathbb T}
\newcommand{\Fin}{\mathcal F}
\newcommand{\id}{\operatorname{id}}
\newcommand{\Int}{\operatorname{int}}

\newcommand{\dist}{\operatorname{dist}}
\DeclareMathOperator*{\IPlim}{IP\text{-}lim}
\DeclareMathOperator{\SIP}{SIP}

\title[Rigid Functions and Mild Mixing]{Rigid Functions, IP-Systems,\\
and Topological Mild Mixing}

\author[S.~Shao]{Song Shao}
\address[S. Shao]{School of Mathematical Sciences, University of Science and
Technology of China, Hefei, Anhui 230026, China}
\email{songshao@ustc.edu.cn}

\author[H.~Xu]{Hui Xu}
\address[H. Xu]{Department of Mathematics, Shanghai Normal University,
Shanghai 200234, China}
\email{huixu@shnu.edu.cn}

\thanks{This research is supported by National Natural Science Foundation of China (12371196, 12426201, 12201599). }


\begin{document}

\begin{abstract}
We study uniform rigidity and topological mild mixing through continuous
observables.  For a fixed sequence of times, the observables rigid along that
sequence form a closed unital $T^{\pm1}$-invariant algebra and determine
the maximal factor uniformly rigid along the prescribed sequence.  We then
give functional forms of the classical $\SIP^{*}$- and
$\operatorname{IP}^{*}$-return-time criteria: a topological dynamical
system is mildly mixing exactly when it has no nonconstant locally
SIP-rigid observable, and in the minimal category the same property is
equivalent to the absence of nonconstant locally IP-rigid observables.
Finally, a locally IP-rigid observable yields a canonical orbit-name factor
carrying marked local data.  For fixed local data, the $T^{\pm1}$-invariant
core of the local rigidity algebra determines a uniformly rigid factor
whenever the core is nontrivial.
\end{abstract}

\maketitle

\section{Introduction}
\label{sec:introduction}

\subsection{Background on IP-systems and mild mixing}
\label{subsec:intro-background}

Finite-sums recurrence is one of the principal links between Ramsey theory
and dynamics.  Given a sequence $(p_i)_{i\geq1}$ of positive integers, put
\[
 p_\alpha=\sum_{i\in\alpha}p_i
 \quad\text{and}\quad
 \operatorname{FS}(p_i)
 =\{p_\alpha:\emptyset\neq\alpha\subseteq\N
   \text{ is finite}\}.
\]
A set containing such a finite-sums set is called an IP-set.  If $T$ is a
homeomorphism, then
\[
 T_\alpha=T^{p_\alpha}
\]
is the associated IP-system.  Furstenberg developed the IP-system formalism
in ergodic theory \cite{FurstenbergIP}; its combinatorial foundation is
Hindman's finite-sums theorem \cite{Hindman}.  See also
\cite{FurstenbergBook} for the recurrence framework connecting these ideas
with dynamics.

Furstenberg and Weiss introduced the term \emph{mild mixing} for the class
of ergodic probability-preserving automorphisms $T$ such that $T\times S$
is ergodic for every conservative ergodic measure-preserving automorphism
$S$, where the measure of the second system may be finite or infinite.
Their main theorem states that this multiplier property is equivalent to
the absence of nontrivial rigid factors; see
\cite[Theorem, p.~128]{FurstenbergWeissFinite}.  Furstenberg subsequently
developed the associated IP-system viewpoint \cite{FurstenbergIP}.  Thus
measurable mild mixing is characterized both by the absence of rigid factors
and by a multiplier property.

Uniform rigidity in topological dynamics was introduced by Glasner and Maon
\cite{GlasnerMaon}.  A topological dynamical system $(X,T)$ is uniformly
rigid if
there is a strictly increasing sequence $(r_k)\subseteq\N$ such that
\[
 \sup_{x\in X}d(T^{r_k}x,x)\longrightarrow0.
\]
Here $d$ is any compatible metric on $X$.
Glasner and Weiss introduced the topological analogue of mild mixing via the
$\SIP^{*}$ hitting-time condition and proved its multiplier characterization;
see \cite[Definition~4.10 and Theorem~4.11]{GlasnerWeiss}.  Independently,
Huang and Ye defined mild mixing by weak disjointness from every transitive
system and developed corresponding complexity and return-time formulations; see
\cite[p.~826 and Theorems~2.3 and~6.6]{HuangYe}.  Akin and Glasner later
gave the positive-time multiplier formulation and the SIP-set realization
theorem in \cite[Theorems~3.6--3.7]{AkinGlasnerSIP}.  In particular,
\[
 \text{topological mixing}\ \Longrightarrow\
 \text{topological mild mixing}\ \Longrightarrow\
 \text{topological weak mixing}.
\]
For minimal systems Huang and Ye proved the stronger characterization that
mild mixing is equivalent to $\operatorname{IP}^{*}$-transitivity:
\[
 N(U,V)\in\operatorname{IP}^{*}
 \quad\text{for all nonempty open }U,V\subseteq X.
\]
See \cite[Theorem~5.8(2)]{HuangYe} and
\cite[Theorem~7.8(2)]{GlasnerYeLocal}.  Polynomial higher-order
consequences, together with a corresponding result for abelian group
actions, were obtained by Cao and Shao
\cite[Theorems~1.3 and~1.6]{CaoShaoPolynomial}.  The
present paper takes a different direction and develops a functional and
factor-theoretic viewpoint.

In ergodic theory, a measure-preserving system is mild mixing if and only if it has no non-trivial rigid factor.  In the topological case  
Glasner and Weiss proved that a topologically mildly mixing system has no
nontrivial uniformly rigid factor
\cite[Lemma~4.14 and Corollary~4.15]{GlasnerWeiss}.  Huang and Ye asked
whether the converse holds in the minimal category
\cite[Appendix, Question~2, p.~845]{HuangYe}.  The question was later
reiterated by Glasner and Ye
\cite[Section~16, p.~352]{GlasnerYeLocal}:
\begin{question}[Huang--Ye]
\label{ques:huang-ye}
If a minimal topological dynamical system has no nontrivial uniformly rigid
factor, must it be mildly mixing?
\end{question}
This question is the main motivation for the functional and factor
constructions below.

\subsection{Our main results}
\label{subsec:main-results}

Throughout the paper, a \emph{topological dynamical system} is a pair
$(X,T)$ in which $X$ is a nonempty compact metric space and $T:X\to X$ is a
homeomorphism; the word \emph{system} will always have this meaning.  We
regard $T$ as generating the $\Nzero$-action $(n,x)\mapsto T^n x$.
All dynamical notions below, including transitivity and mild mixing, refer to
this action.  Accordingly, IP-, SIP-, and hitting-time sets are positive-time
subsets of $\N=\{1,2,\ldots\}$.  Since $T$ is a homeomorphism, negative
iterates are also available and will be used explicitly in invariant
algebras and bilateral orbit-name factors.  Write
$C(X)=C(X,\mathbb R)$, and call its elements continuous observables; the
notation $C(X,[0,1])$ has the analogous meaning.  For a bounded function
$h$ and a nonempty set $E\subseteq X$, put
\[
 \|h\|_E=\sup_{x\in E}|h(x)|.
\]
In particular, $\|h\|_\infty=\|h\|_X$.  Let $\Fin$ denote the collection
of all nonempty finite subsets of $\N$, ordered by
$\eta<\alpha$ when $\max\eta<\min\alpha$.
A family $\mathcal G\subseteq C(X)$ is \emph{point-separating} if, for every
distinct $x,y\in X$, some $f\in\mathcal G$ satisfies $f(x)\neq f(y)$.
The symbol $\mathbf1$ denotes the constant-one function, while
$\mathbf1_A$ denotes the indicator of a set $A$.  All algebras below are
real subalgebras of $C(X)$, and \emph{unital} means that they contain
$\mathbf1$.  An algebra $\mathcal A$ is $T^{\pm1}$-invariant if
$f\circ T^j\in\mathcal A$ for every $f\in\mathcal A$ and $j\in\Z$.
The notation $\IPlim_{\alpha\in\Fin}$ denotes convergence along
the full tails $\{\alpha\in\Fin:\alpha>\eta\}$; further IP conventions
are recalled in Subsection~\ref{subsec:ip-systems}.

For a strictly increasing sequence $\mathbf r=(r_k)\subseteq\N$, let
\[
 \mathcal A_{\mathbf r}(X,T)
 =\{f\in C(X):
   \|f\circ T^{r_k}-f\|_\infty\longrightarrow0\}.
\]

\begin{thma}
\label{thmA:intro}
Let $(X,T)$ be a topological dynamical system.  Then $(X,T)$ is uniformly
rigid if and only if there is a strictly increasing sequence
$(r_k)\subseteq\N$ and a countable point-separating family
$\mathcal G\subseteq C(X)$ such that
\[
 \|f\circ T^{r_k}-f\|_\infty\longrightarrow0
 \qquad \forall f\in\mathcal G.
\]
For a fixed such sequence $\mathbf r$, the algebra
$\mathcal A_{\mathbf r}(X,T)$ is a closed unital
$T^{\pm1}$-invariant subalgebra of $C(X)$.  It determines the maximal
factor of $(X,T)$ which is uniformly rigid along $\mathbf r$.
\end{thma}

The precise factor statement and its proof are given in
Theorem~\ref{thm:maximal-rigid-factor}.

Let $(r_i)\subseteq\N$ be strictly increasing.  A continuous function is
\emph{locally SIP-rigid} if it returns uniformly to itself on one nonempty
open set along all sufficiently late symmetric finite-sums differences.

\begin{thma}
\label{thmB:intro}
A topological dynamical system is mildly mixing if and only if it has no
nonconstant locally SIP-rigid continuous function.
\end{thma}

Theorem~\ref{thmB:intro} is proved as
Theorem~\ref{thm:mildmixing-sip-functions}.  Its proof is a functional
dualization of the Glasner--Weiss $\SIP^{*}$ criterion.

For a minimal system, symmetric finite-sums differences may be replaced by
ordinary finite sums.  A continuous function $f$ is \emph{locally
IP-rigid} if, for some nonempty open $U$ and some sequence
$(p_i)\subseteq\N$,
\[
 \IPlim_{\alpha\in\Fin}
 \|f\circ T^{p_\alpha}-f\|_U=0.
\]

\begin{thma}
\label{thmC:intro}
A minimal topological dynamical system is mildly mixing if and only if it has
no nonconstant locally IP-rigid continuous function.
\end{thma}

This is Theorem~\ref{thm:minimal-ip-functions}.  Theorems~\ref{thmB:intro}
and~\ref{thmC:intro} are
functional dualizations of the known $\SIP^{*}$- and
$\operatorname{IP}^{*}$-return-time criteria.  The factor structure
carried by the resulting local observables is described next.

For a nonempty open set $U\subseteq X$ and a sequence
$\mathbf p=(p_i)\subseteq\N$, put
\[
 \mathcal L(U,\mathbf p)
 =\left\{f\in C(X):
   \IPlim_{\alpha\in\Fin}
   \|f\circ T^{p_\alpha}-f\|_U=0\right\}
\]
and let $\mathcal C(U,\mathbf p)$ be the largest
$T^{\pm1}$-invariant subalgebra of $\mathcal L(U,\mathbf p)$.
Here $\mathbb R\mathbf1$ denotes the algebra of real constant
functions.

\begin{thma}
\label{thmD:intro}
Let $(X,T)$ be minimal.  Every nonconstant locally IP-rigid observable
determines a canonical nontrivial factor carrying marked local IP-rigidity
data.  For every pair $(U,\mathbf p)$, the invariant core
$\mathcal C(U,\mathbf p)$ determines the maximal factor whose pullback
algebra is contained in $\mathcal L(U,\mathbf p)$; this factor is uniformly
rigid whenever the core is nontrivial.  Moreover,
\begin{align*}
 (X,T)\text{ is not mildly mixing}
 &\Longleftrightarrow
 \mathcal L(U,\mathbf p)\neq\mathbb R\mathbf1
 \text{ for some }(U,\mathbf p),\\
 (X,T)\text{ has a nontrivial uniformly rigid factor}
 &\Longleftrightarrow
 \mathcal C(U,\mathbf p)\neq\mathbb R\mathbf1
 \text{ for some }(U,\mathbf p).
\end{align*}
\end{thma}

The marked-factor assertion is Theorem~\ref{thm:canonical-marked-factor},
the invariant-core construction is Theorem~\ref{thm:core-factor}, and the
two equivalences are Corollary~\ref{cor:exact-core-form}.
Theorem~\ref{thmD:intro} does not
resolve Question~\ref{ques:huang-ye}; rather, it identifies the unresolved
step as the passage from a nontrivial local rigidity algebra to a nontrivial
invariant core.

\subsection*{Organization of the paper}
The paper is organized as follows.  Section~\ref{sec:preliminaries} recalls
the families of positive integers, topological dynamics, IP-systems, and the classical
mild-mixing criteria used later.  Section~\ref{sec:theorem-a} proves
Theorem~\ref{thmA:intro} and studies factors generated by globally rigid
observables.  Section~\ref{sec:theorem-b} proves Theorem~\ref{thmB:intro}
and explains precisely why
globally rigid functions alone do not presently characterize minimal mild
mixing.  Section~\ref{sec:theorem-c} proves Theorem~\ref{thmC:intro}.
Section~\ref{sec:factors} proves Theorem~\ref{thmD:intro} and reformulates
Question~\ref{ques:huang-ye} in terms of invariant cores.

\section{Preliminaries}
\label{sec:preliminaries}

This section fixes the notation and background used throughout the paper.  We
review the relevant families of positive integers, the standing conventions
for topological dynamical systems, the IP-system formalism, and the classical
return-time criteria for topological mild mixing.

\subsection{Families of positive integers}
\label{subsec:families}

Throughout,
\[
 \N=\{1,2,\ldots\},\qquad
 \Nzero=\{0,1,2,\ldots\},\qquad
 \Z=\{\ldots,-1,0,1,\ldots\},
\]
and $\Fin$ denotes the collection of all nonempty finite subsets of
$\N$.  

Every dynamical system in this paper carries the $\Nzero$-action
generated by its defining homeomorphism, while all Furstenberg families and
hitting-time sets below consist of positive times and are therefore families
of subsets of $\N$.  Whenever $\Z$ occurs in dynamical notation, it refers
to inverse iterates or bilateral orbit coordinates, not to the acting
semigroup.  A \emph{Furstenberg family} is a nonempty proper collection
$\mathcal G\subseteq\mathcal P(\N)$ that is upward hereditary.  Its dual is
\[
 \mathcal G^{*}
 =\{A\subseteq\N:A\cap G\neq\emptyset
     \text{ for every }G\in\mathcal G\}.
\]
See \cite{AkinRecurrence,FurstenbergBook} for the general family formalism.

For $\alpha,\beta\in\Fin$, write $\alpha<\beta$ when
$\max\alpha<\min\beta$.  A sequence
$\alpha_1<\alpha_2<\cdots$ is called a \emph{block sequence}.

For a sequence $(p_i)_{i\geq1}\subseteq\N$, put
\[
 p_\alpha=\sum_{i\in\alpha}p_i,
 \qquad
 \operatorname{FS}(p_i)=\{p_\alpha:\alpha\in\Fin\}.
\]
For $M\in\Nzero$, the notation
$\operatorname{FS}(p_i:i>M)$ means that only indices greater than $M$ are
used in the finite sums.
A set $F\subseteq\N$ is an \emph{IP-set} if it contains
$\operatorname{FS}(p_i)$ for some sequence $(p_i)\subseteq\N$.  Unless
explicitly stated otherwise, no monotonicity or injectivity is imposed on an
IP-generating sequence.  The dual of the family of IP-sets is denoted by
$\operatorname{IP}^{*}$.  Thus
\[
 A\in\operatorname{IP}^{*}
 \quad\Longleftrightarrow\quad
 A\cap F\neq\emptyset\text{ for every IP-set }F.
\]
The basic references are \cite{FurstenbergBook,Hindman}.

Let $(r_i)_{i\geq1}\subseteq\N$ be strictly increasing.  For
$M\in\Nzero$, define
\[
 \operatorname{FS}_0(r_i:i>M)
 =\{0\}\cup
 \left\{\sum_{i\in\alpha}r_i:
   \emptyset\neq\alpha\subseteq\{M+1,M+2,\ldots\}
   \text{ is finite}\right\}
\]
and
\begin{equation}
 \SIP(r_i:i>M)
 =\{a-b>0:
   a,b\in\operatorname{FS}_0(r_i:i>M)\}.
 \label{eq:tail-sip}
\end{equation}
A set containing $\SIP(r_i:i>0)$ for some strictly increasing sequence
$(r_i)\subseteq\N$ is an \emph{SIP-set}.  Its dual family is denoted by
$\SIP^{*}$.  We use the
normalization of Glasner--Weiss and Akin--Glasner
\cite[Section~4]{GlasnerWeiss} and
\cite[Section~1]{AkinGlasnerSIP}.

Every SIP-set is an IP-set, because in
\eqref{eq:tail-sip} one may take $b=0$; hence
$\operatorname{FS}(r_i)\subseteq\SIP(r_i:i>0)$.  Consequently,
\[
 \operatorname{IP}^{*}\subseteq\SIP^{*}.
\]

\subsection{Basics of topological dynamics}
\label{subsec:top-dynamics}

A \emph{topological dynamical system} is a pair $(X,T)$, where $X$ is a
nonempty compact metric space and $T:X\to X$ is a homeomorphism.  We write
$(X,T)$ for the $\Nzero$-action generated by $T$, and use \emph{system} as
shorthand for a topological dynamical system.  Thus the action iterates are
$T^n$ with $n\in\Nzero$.  Invertibility also makes $T^j$ available for
$j\in\Z$, and such inverse iterates will always be indicated explicitly.
Whenever metric notation is used, $d$ denotes a fixed compatible metric on
the relevant space.  For nonempty sets $A,B\subseteq X$, write
\[
 d(x,A)=\inf_{a\in A}d(x,a),
 \qquad
 \dist(A,B)=\inf_{a\in A,\,b\in B}d(a,b),
\]
and let $\Int A$ denote the interior of $A$.  We use $\Torus$ for the circle
group, identified with both $\mathbb R/\mathbb Z$ and the unit circle in
$\mathbb C$, and write $\|t\|_{\Torus}$ for the distance from
$t\in\mathbb R/\mathbb Z$ to zero.  A real-valued function $h$ on $X$ is
\emph{Lipschitz} if there is $L\geq0$ such that
$|h(x)-h(y)|\leq Ld(x,y)$ for all $x,y\in X$.  A system is
\emph{nontrivial} if its underlying space has more than one point.  It is
\emph{transitive} if for all nonempty open $U,V\subseteq X$ there is
$n\in\N$ such that $U\cap T^{-n}V\neq\emptyset$, and it is
\emph{minimal} if every forward orbit $\{T^n x:n\in\Nzero\}$ is dense.
For a homeomorphism this is equivalent to density of every two-sided orbit.
A \emph{factor map}
\[
 \pi:(X,T)\longrightarrow(Y,S)
\]
is a continuous surjection satisfying $\pi T=S\pi$.  This identity holds for
every $n\in\Nzero$; since $T$ and $S$ are homeomorphisms, it in fact holds
for every $n\in\Z$.  A factor
map, and its target factor, are \emph{nontrivial} if the target space has more
than one point.  A \emph{conjugacy} is a homeomorphism intertwining the two
transformations.

For a factor map $q:X\to Y$, write
\[
 R_q=\{(x,x')\in X^2:q(x)=q(x')\}.
\]
Its pullback algebra is
\[
 q^{*}C(Y)=\{h\circ q:h\in C(Y)\}\subseteq C(X).
\]
We say that a factor $q_1:X\to Y_1$ is \emph{larger} than a factor
$q_2:X\to Y_2$ if $q_2$ factors through $q_1$.  Maximal and minimal
factors below are understood in this order.  Standard references include
\cite{AkinRecurrence,Auslander}.

For nonempty open sets $U,V\subseteq X$, put
\[
 N(U,V)=\{n\in\N:U\cap T^{-n}V\neq\emptyset\}.
\]
Thus $N(U,V)$ records the positive hitting times of the $\Nzero$-action.
The product of $(X,T)$ and $(Y,S)$ is
$(X\times Y,T\times S)$.  Two systems are \emph{weakly disjoint} if their
product is transitive.  A system is \emph{weakly mixing} if its self-product
is transitive, and it is
\emph{topologically mixing} if $N(U,V)$ is cofinite in $\N$, that is, if its
complement is finite, for every pair of nonempty open sets $U,V\subseteq X$.
A system is \emph{mildly mixing}
if its product with every transitive topological dynamical system is
transitive.

For continuous maps $F,G:X\to X$, write
\[
 D(F,G)=\sup_{x\in X}d(Fx,Gx).
\]
For $m,n\in\Nzero$, commutativity of the powers of $T$ and surjectivity of
$T^m$ give
\[
 D(T^{m+n},T^m)=D(T^n,\id_X).
\]
Consequently,
\begin{equation}
 D(T^{m+n},\id_X)
 \leq D(T^m,\id_X)+D(T^n,\id_X),
 \label{eq:uniform-subadditivity}
\end{equation}
and iteration gives the corresponding estimate for every finite sum of
nonnegative powers.
A strictly increasing sequence $(r_k)\subseteq\N$ is a
\emph{uniform rigidity sequence} for $(X,T)$ if
\[
 D(T^{r_k},\id_X)\longrightarrow0.
\]
The system is \emph{uniformly rigid} if it admits such a sequence.  This
notion is independent of the compatible metric on $X$.
A \emph{uniformly rigid factor} is a factor whose target system is uniformly
rigid.

\subsection{Basics of IP-systems}
\label{subsec:ip-systems}

Let $S_1,S_2,\ldots$ be commuting homeomorphisms of $X$.  For
$\alpha=\{i_1<\cdots<i_k\}\in\Fin$, define
\[
 T_\alpha=S_{i_1}\cdots S_{i_k}.
\]
Then
\[
 T_{\alpha\cup\beta}=T_\alpha T_\beta
 \quad\text{whenever }\alpha\cap\beta=\emptyset.
\]
The family $(T_\alpha)_{\alpha\in\Fin}$ is an \emph{IP-system}; see \cite{FurstenbergBook, FurstenbergIP}.  The main example in this paper is obtained from a
single homeomorphism $T$ and a sequence $(p_i)\subseteq\N$:
\[
 T_\alpha=T^{p_\alpha}.
\]
This is an IP-subfamily of the $\Nzero$-action generated by $T$.

Let $(x_\alpha)_{\alpha\in\Fin}$ be a net in $X$.  We write
\[
 \IPlim_{\alpha\in\Fin}x_\alpha=x
\]
if, for every neighborhood $O$ of $x$, there is $\eta\in\Fin$ such
that $x_\alpha\in O$ whenever $\alpha>\eta$.  We refer to this as the
\emph{full-tail IP-limit}.

An associated IP-system $T_\alpha=T^{p_\alpha}$ is
\emph{tail equicontinuous} if, for every $\varepsilon>0$, there are
$\delta>0$ and $\eta\in\Fin$ such that
\[
 d(x,y)<\delta
 \quad\Longrightarrow\quad
 d(T^{p_\beta}x,T^{p_\beta}y)<\varepsilon
 \quad(\beta>\eta).
\]

\subsection{Classical return-time criteria for mild mixing}
\label{subsec:classical-mild-mixing}

The following criterion is the $\SIP^{*}$ definition of Glasner and Weiss
together with their multiplier theorem
\cite[Definition~4.10 and Theorem~4.11]{GlasnerWeiss}.  Huang and Ye give
the equivalent $(\operatorname{IP}-\operatorname{IP})^{*}$ return-time
form in \cite[Theorem~6.6]{HuangYe}.  The positive-time multiplier
formulation and the required SIP-set realization are
\cite[Theorems~3.6--3.7]{AkinGlasnerSIP}.

\begin{lemma}[Glasner--Weiss; Huang--Ye; Akin--Glasner]
\label{lem:sip-multiplier}
A topological dynamical system $(X,T)$ is mildly mixing if and only if
\[
 N(U,V)\in\SIP^{*}
\]
for every pair of nonempty open sets $U,V\subseteq X$.
\end{lemma}

The realization theorem is essential here: the criterion does not follow
merely from the fact that self-return sets in transitive systems are
SIP-sets.

The next criterion is precisely
\cite[Theorem~5.8(2)]{HuangYe}; see also
\cite[Theorem~7.8(2)]{GlasnerYeLocal}.
\begin{lemma}[Huang--Ye]
\label{lem:minimal-ipstar}
A minimal topological dynamical system $(X,T)$ is mildly mixing if and only
if
\[
 N(U,V)\in\operatorname{IP}^{*}
\]
for every pair of nonempty open sets $U,V\subseteq X$.
\end{lemma}

The following result is \cite[Lemma~4.14 and
Corollary~4.15]{GlasnerWeiss}.
\begin{lemma}[Glasner--Weiss]
\label{lem:mild-no-uniform-factor}
A nontrivial mildly mixing system is not uniformly rigid.  More generally,
a mildly mixing system has no nontrivial uniformly rigid factor.
\end{lemma}

The factor statement also follows from the first assertion because mild
mixing passes to factors.

\section{Rigid functions and uniform rigidity}
\label{sec:theorem-a}

This section develops the global theory of rigid functions.  We first recall
the factor--algebra correspondence, then construct the maximal factor that is
uniformly rigid along a prescribed sequence, describe the factor generated by
a single rigid observable, and relate uniform rigidity to equicontinuity
along IP-sets.

\subsection{Invariant algebras and their factors}
\label{subsec:invariant-algebras}

Let $U_T:C(X)\to C(X)$ be the Koopman isometry $U_Tf=f\circ T$.  The
following general lemma will be used both for the global rigidity algebra and
for the local rigidity algebra in Section~\ref{sec:factors}.

\begin{lemma}%[Factor determined by an invariant algebra]
\label{lem:factor-algebra}
Let $(X,T)$ be a topological dynamical system, and let
$\mathcal A\subseteq C(X)$ be a closed unital algebra satisfying
\[
 U_T\mathcal A=\mathcal A.
\]
Define
\[
 x\sim_{\mathcal A}y
 \quad\Longleftrightarrow\quad
 f(x)=f(y)\ \text{for every }f\in\mathcal A.
\]
Then there are a topological dynamical system
$(Y_{\mathcal A},S_{\mathcal A})$ and a factor map
\[
 q_{\mathcal A}:(X,T)\longrightarrow(Y_{\mathcal A},S_{\mathcal A})
\]
such that
\[
 q_{\mathcal A}^{*}C(Y_{\mathcal A})=\mathcal A
 \quad\text{and}\quad
 R_{q_{\mathcal A}}=\sim_{\mathcal A}.
\]
Moreover, if $q:(X,T)\to(Z,R)$ is a factor satisfying
$q^{*}C(Z)\subseteq\mathcal A$, there is a unique factor map
$\theta:(Y_{\mathcal A},S_{\mathcal A})\to(Z,R)$ such that
\[
 q=\theta\circ q_{\mathcal A}.
\]
\end{lemma}

\begin{proof}
Since $C(X)$ is separable, so is the closed subspace $\mathcal A$.
Choose a sequence $(f_m)_{m\geq1}$ dense in the closed unit ball of
$\mathcal A$, and define
\[
 q_{\mathcal A}(x)=(f_m(x))_{m\geq1}
 \in[-1,1]^{\N}.
\]
Equip the product with
\[
 \rho(a,b)=\sum_{m=1}^{\infty}2^{-m}|a_m-b_m|.
\]
A finite-coordinate argument shows that $q_{\mathcal A}$ is continuous.
Its image $Y_{\mathcal A}$ is compact and metrizable.  Density of
$(f_m)$ gives
\[
 q_{\mathcal A}(x)=q_{\mathcal A}(y)
 \quad\Longleftrightarrow\quad
 x\sim_{\mathcal A}y.
\]

The equality $U_T\mathcal A=\mathcal A$ implies invariance under both
$U_T$ and $U_T^{-1}$.  Hence
\[
 S_{\mathcal A}q_{\mathcal A}(x)=q_{\mathcal A}(Tx)
\]
defines a bijection of $Y_{\mathcal A}$, and
\[
 S_{\mathcal A}^{-1}q_{\mathcal A}(x)=q_{\mathcal A}(T^{-1}x)
\]
defines its continuous inverse.  Thus $S_{\mathcal A}$ is a
homeomorphism and $q_{\mathcal A}$ is a factor map.

Every member of $\mathcal A$ is constant on the fibers of
$q_{\mathcal A}$ and therefore factors continuously through the quotient.
Conversely, the coordinate functions on $Y_{\mathcal A}$ separate points.
The real Stone--Weierstrass theorem and closedness of $\mathcal A$ give
\[
 q_{\mathcal A}^{*}C(Y_{\mathcal A})=\mathcal A.
\]

Finally, if $q^{*}C(Z)\subseteq\mathcal A$ and
$q_{\mathcal A}(x)=q_{\mathcal A}(y)$, then
$h(qx)=h(qy)$ for every $h\in C(Z)$, whence $qx=qy$.
Thus $q$ is constant on the fibers of $q_{\mathcal A}$.  The quotient
property gives a unique continuous surjection $\theta$ satisfying
$q=\theta q_{\mathcal A}$, and the intertwining identity follows from
surjectivity of $q_{\mathcal A}$.
\end{proof}

\subsection{The rigidity algebra and proof of Theorem~\ref{thmA:intro}}
\label{subsec:rigidity-algebra}

\begin{definition}
\label{def:rigid-function}
Let $\mathbf r=(r_k)_{k\geq1}\subseteq\N$ be strictly increasing.  A
function $f\in C(X)$ is \emph{$\mathbf r$-rigid} if
\[
 \|f\circ T^{r_k}-f\|_\infty\longrightarrow0.
\]
It is \emph{rigid} if it is $\mathbf r$-rigid for some such sequence.
The corresponding \emph{rigidity algebra} is
\[
 \mathcal A_{\mathbf r}(X,T)
 =\{f\in C(X):
   \|f\circ T^{r_k}-f\|_\infty\to0\}.
\]
A family of functions is \emph{jointly rigid} if all its members are rigid
along one common sequence.
\end{definition}

\begin{theorem}%[Maximal factor determined by a rigidity sequence]
\label{thm:maximal-rigid-factor}
Let $(X,T)$ be a topological dynamical system and let
$\mathbf r=(r_k)\subseteq\N$ be strictly increasing.  Then
$\mathcal A_{\mathbf r}(X,T)$ is a closed unital
$T^{\pm1}$-invariant subalgebra of $C(X)$.  It determines a factor
\[
 \pi_{\mathbf r}:(X,T)\longrightarrow
 (Y_{\mathbf r},S_{\mathbf r})
\]
such that
\[
 \pi_{\mathbf r}^{*}C(Y_{\mathbf r})
 =\mathcal A_{\mathbf r}(X,T).
\]
The sequence $\mathbf r$ is a uniform rigidity sequence for
$(Y_{\mathbf r},S_{\mathbf r})$, and this factor is maximal among factors
uniformly rigid along $\mathbf r$.

Moreover, the following are equivalent:
\begin{enumerate}[label=\textup{(\roman*)}]
\item $\mathbf r$ is a uniform rigidity sequence for $(X,T)$;
\item $\mathcal A_{\mathbf r}(X,T)=C(X)$;
\item $\mathcal A_{\mathbf r}(X,T)$ separates the points of $X$;
\item there is a countable point-separating family in $C(X)$ jointly rigid
      along $\mathbf r$.
\end{enumerate}
\end{theorem}

\begin{proof}
Constants belong to the rigidity algebra, and it is plainly closed under
addition and scalar multiplication.  If $f,g$ belong to it, then
\[
 \|(fg)\circ T^{r_k}-fg\|_\infty
 \leq
 \|f\|_\infty\|g\circ T^{r_k}-g\|_\infty+\|g\|_\infty\|f\circ T^{r_k}-f\|_\infty,
\]
so it is an algebra.  If $f_n\to f$ uniformly and every $f_n$ is
$\mathbf r$-rigid, then
\[
 \|f\circ T^{r_k}-f\|_\infty
 \leq2\|f-f_n\|_\infty+
 \|f_n\circ T^{r_k}-f_n\|_\infty,
\]
which proves closedness.

For every $j\in\Z$, surjectivity of $T^j$ gives the exact identity
\begin{equation}
 \|(f\circ T^j)\circ T^{r_k}-f\circ T^j\|_\infty
 =\|f\circ T^{r_k}-f\|_\infty.
 \label{eq:translate-rigidity}
\end{equation}
Thus $U_T\mathcal A_{\mathbf r}=\mathcal A_{\mathbf r}$.  Apply
Lemma~\ref{lem:factor-algebra} to obtain the factor
$\pi_{\mathbf r}:X\to Y_{\mathbf r}$.  Choose $(f_m)$ dense in the closed unit ball of
$\mathcal A_{\mathbf r}$, identify the factor with the coordinate
realization furnished by Lemma~\ref{lem:factor-algebra}, and equip it with
\[
 \rho(y,z)=\sum_{m=1}^{\infty}2^{-m}|y_m-z_m|.
\]
Then
\[
\begin{split}
 \sup_{y\in Y_{\mathbf r}}
 \rho(S_{\mathbf r}^{r_k}y,y)
 &\leq
 \sum_{m=1}^{\infty}2^{-m}
 \|f_m\circ T^{r_k}-f_m\|_\infty
 \longrightarrow0.
\end{split}
\]
Indeed, first control finitely many coordinates and then use the summable
tail.  Hence $\mathbf r$ is a uniform rigidity sequence for the factor.

Let $q:(X,T)\to(Z,R)$ be a factor for which
$R^{r_k}\to\id_Z$ uniformly.  For every $h\in C(Z)$, uniform continuity
of $h$ yields
\[
 \|h\circ q\circ T^{r_k}-h\circ q\|_\infty
 =\|h\circ R^{r_k}\circ q-h\circ q\|_\infty
 \longrightarrow0.
\]
Thus $q^{*}C(Z)\subseteq\mathcal A_{\mathbf r}$, and
Lemma~\ref{lem:factor-algebra} makes $q$ factor uniquely through
$\pi_{\mathbf r}$.  This proves maximality.

If $T^{r_k}\to\id_X$ uniformly, uniform continuity makes every
$f\in C(X)$ $\mathbf r$-rigid, proving
\textup{(i)}$\Rightarrow$\textup{(ii)}.  The implication
\textup{(ii)}$\Rightarrow$\textup{(iii)} is immediate.  For
\textup{(iii)}$\Rightarrow$\textup{(iv)}, choose a countable dense subset
of the closed unit ball of $\mathcal A_{\mathbf r}$.  It still separates
points: if $f(x)\neq f(y)$, an approximant within
$\lvert f(x)-f(y)\rvert/3$ of $f$ separates $x$ and $y$.
Under \textup{(iv)},
$\pi_{\mathbf r}$ is injective, hence a conjugacy.  Uniform rigidity of
the factor therefore proves \textup{(i)}. The proof is completed.
\end{proof}

Theorem~\ref{thm:maximal-rigid-factor} proves
Theorem~\ref{thmA:intro}.

\subsection{Factors generated by one rigid observable}
\label{subsec:single-rigid-observable}

The homeomorphism hypothesis requires bilateral orbit names.  For
$f\in C(X,[0,1])$, define
\[
 \Phi_f(x)=\bigl(f(T^j x)\bigr)_{j\in\Z}
 \in[0,1]^{\Z},
 \qquad
 X_f=\Phi_f(X),
\]
and let $\sigma$ be the bilateral shift, with
$(\sigma y)_j=y_{j+1}$.  Equip the full product with
\begin{equation}
 \rho_{\pm}(a,b)
 =\frac13\sum_{j\in\Z}2^{-|j|}|a_j-b_j|.
 \label{eq:bilateral-metric}
\end{equation}
Any nonconstant real-valued continuous function can be affinely normalized
to have range in $[0,1]$, without changing any of the rigidity
convergences below.

\begin{proposition}%[Bilateral orbit-name factor]
\label{prop:bilateral-rigid-factor}
Let $f\in C(X,[0,1])$ and let
$\mathbf r=(r_k)\subseteq\N$ be strictly increasing.
\begin{enumerate}[label=\textup{(\roman*)}]
\item $\Phi_f:(X,T)\to(X_f,\sigma)$ is a factor map in the
      category of homeomorphisms.
\item $f$ is $\mathbf r$-rigid if and only if $\mathbf r$ is a
      uniform rigidity sequence for $(X_f,\sigma)$.
\item The factor $X_f$ is the smallest factor through which $f$
      factors.
\item The system $(X,T)$ has a nonconstant rigid continuous function if
      and only if it has a nontrivial uniformly rigid factor.
\end{enumerate}
\end{proposition}

\begin{proof}
The map $\Phi_f$ is continuous and satisfies
\[
 \Phi_fT=\sigma\Phi_f.
\]
Its compact image is invariant under both $\sigma$ and $\sigma^{-1}$,
because $T$ is a homeomorphism.  This proves \textup{(i)}.

Put
\[
 \varepsilon_k=\|f\circ T^{r_k}-f\|_\infty.
\]
For every $j\in\Z$, equation~\eqref{eq:translate-rigidity} gives
\[
 \|f\circ T^{j+r_k}-f\circ T^j\|_\infty=\varepsilon_k.
\]
Since the weights in \eqref{eq:bilateral-metric} sum to one,
\[
 \sup_{y\in X_f}\rho_{\pm}(\sigma^{r_k}y,y)
 \leq\varepsilon_k.
\]
Conversely, the zeroth coordinate has weight $1/3$, so
\[
 |f(T^{r_k}x)-f(x)|
 \leq3\rho_{\pm}
 \bigl(\sigma^{r_k}\Phi_f(x),\Phi_f(x)\bigr).
\]
This proves \textup{(ii)}.

If $q:(X,T)\to(Z,R)$ is a factor and $f=h\circ q$, define
\[
 \Psi(z)=\bigl(h(R^jz)\bigr)_{j\in\Z}.
\]
Then $\Psi$ is a factor map from $Z$ onto $X_f$ and
$\Phi_f=\Psi q$, which proves \textup{(iii)}.

If $f$ is nonconstant and rigid, \textup{(ii)} gives a nontrivial uniformly
rigid factor.  Conversely, if $q:X\to Z$ is a nontrivial uniformly rigid
factor, choose a nonconstant $h\in C(Z,[0,1])$.  Then $h\circ q$ is a
nonconstant rigid function.  This proves \textup{(iv)}.
\end{proof}

A function $f$ is \emph{topologically generating} if
$\Phi_f$ is injective, equivalently if
$\{f\circ T^j:j\in\Z\}$ separates points.

\begin{corollary}
\label{cor:generating-rigid}
If $f$ is topologically generating, then $f$ is
$\mathbf r$-rigid if and only if $\mathbf r$ is a uniform rigidity
sequence for $(X,T)$.
\end{corollary}

\begin{proof}
The factor map $\Phi_f$ is then a conjugacy, so the result follows
from Proposition~\ref{prop:bilateral-rigid-factor}(ii).
\end{proof}

\begin{remark}
\label{rem:one-function-not-whole-system}
A single rigid observable need not make the original system uniformly rigid.
Let $R_\alpha:x\mapsto x+\alpha$ be an irrational circle rotation, let
$\sigma:\{0,1\}^{\Z}\to\{0,1\}^{\Z}$ be the two-sided full shift, and consider
$R_\alpha\times\sigma$.  Every nonconstant function depending only on the
circle coordinate is rigid along a sequence $(q_k)\subseteq\N$ satisfying
$q_k\to\infty$ and
$\|q_k\alpha\|_{\Torus}\to0$.  The product is not uniformly rigid,
because the full shift is a factor and is not uniformly rigid.  Recall that a
homeomorphism is \emph{expansive} if there
is $c>0$ such that
\[
 d(T^nx,T^ny)<c\quad(n\in\Z)
 \quad\Longrightarrow\quad x=y;
\]
such a number $c$ is an expansive constant.  Indeed, if
$r\in\N$ and $D(T^r,\id_X)<c$, then
\[
 d(T^nT^r x,T^nx)<c\qquad(x\in X,\ n\in\Z),
\]
so expansiveness forces $T^r=\id$; this is impossible for a positive power
of the full shift.  The rigid observable detects the rotation factor, exactly
as Proposition~\ref{prop:bilateral-rigid-factor} predicts.
\end{remark}

\subsection{Uniform rigidity and tail equicontinuity}
\label{subsec:uniform-tail-equicontinuity}

For completeness, we record the IP-system formulation related to
Theorem~\ref{thmA:intro}.  For a set $A\subseteq\N$, $A$-equicontinuity
means equicontinuity of the family $\{T^n:n\in A\}$.  Huang and Ye \cite[Lemma~4.1]{HuangYe} proved the following result:

\begin{proposition}
\label{prop:uniform-iff-tail-equicont}
A topological dynamical system is uniformly rigid if and only if there is a
sequence $(p_i)\subseteq\N$ for which the associated IP-system
$T_\alpha=T^{p_\alpha}$ is tail equicontinuous.
\end{proposition}

\section{Locally SIP-rigid functions and mild mixing}
\label{sec:theorem-b}

This section develops the functional characterization of mild mixing.  After
recalling the eigenfunction analogy and explaining why global rigidity is not
the appropriate local criterion, we introduce locally SIP-rigid functions and
prove that their absence is equivalent to topological mild mixing.

\subsection{The eigenfunction analogy}
\label{subsec:eigenfunction-analogy}

Continuous eigenfunctions provide the classical functional obstruction to
weak mixing.  The following is the single-homeomorphism case of
\cite[Corollary~2.11]{KeynesRobertson}.

\begin{lemma}[Keynes--Robertson]
\label{lem:keynes-robertson}
Let $(X,T)$ be a minimal topological dynamical system.  Then $(X,T)$ is
weakly mixing if and only if every continuous function
$g:X\to\mathbb C$ satisfying
\[
 g\circ T=\lambda g
 \qquad(\lambda\in\Torus)
\]
is constant.
\end{lemma}

The eigenfunction equation is global and uses a single time.  Mild mixing is
instead controlled by whole symmetric IP tails, so its exact functional
obstruction must be localized.

\subsection{Why global rigid functions are not the local criterion}
\label{subsec:global-versus-local}

Proposition~\ref{prop:bilateral-rigid-factor}(iv) and
Lemma~\ref{lem:mild-no-uniform-factor} give
\[
 \text{mild mixing}
 \quad\Longrightarrow\quad
 \text{every globally rigid continuous function is constant}.
\]
Conversely, the absence of nonconstant globally rigid functions is equivalent
to the absence of nontrivial uniformly rigid factors.  In the minimal
category, whether this implies mild mixing is exactly
Question~\ref{ques:huang-ye}.  Thus a global rigid observable does not
\emph{presently} give a known characterization of minimal mild mixing; we do
not assert that the converse fails for minimal systems.

Remark~\ref{rem:one-function-not-whole-system} gives a different, elementary
distinction: one rigid function can detect a proper uniformly rigid factor
without making the whole system uniformly rigid.  Outside the minimal
category, the failure of a global-function characterization is explicit.

There are also nontrivial minimal systems in which global rigid functions
abound: Glasner and Maon constructed minimal systems that are simultaneously
weakly mixing and uniformly rigid
\cite[Proposition~6.5]{GlasnerMaon}.  Along a uniform rigidity sequence every
continuous function on such a system is rigid.
Lemma~\ref{lem:mild-no-uniform-factor} shows that these examples are not
mildly mixing.  Thus replacing ``weakly mixing'' by ``mildly mixing'' in the
eigenfunction analogy requires a genuinely stronger obstruction than the
absence of eigenfunctions.

\begin{example}[No rigid observable, but not mildly mixing]
\label{ex:compactified-integers}
Let
\[
 X_\infty=\Z\cup\{\infty\}
\]
be the one-point compactification of the discrete integers, and define the
homeomorphism
\[
 T(n)=n+1,\qquad T(\infty)=\infty.
\]
The system is not mildly mixing.  Indeed, $\{0\}$ is open and
\[
 N(\{0\},\{0\})=\emptyset,
\]
so the $\SIP^{*}$ condition fails.

Nevertheless, $X_\infty$ has no nonconstant rigid continuous function.
Suppose $f\in C(X_\infty)$ and
\[
 \|f\circ T^{r_k}-f\|_\infty\longrightarrow0
\]
for some strictly increasing sequence $(r_k)\subseteq\N$.  Fix $m\in\Z$
and put $n_k=m-r_k$.  Then
$|n_k|\to\infty$, and hence $n_k$ converges to the point $\infty$
in the one-point compactification.  Therefore
\[
 f(n_k)\longrightarrow f(\infty).
\]
On the other hand,
\[
 |f(m)-f(n_k)|
 =|f(T^{r_k}n_k)-f(n_k)|
 \leq\|f\circ T^{r_k}-f\|_\infty\longrightarrow0.
\]
Thus $f(m)=f(\infty)$ for every $m\in\Z$, and $f$ is constant.
\end{example}

The compactified-integers system also appears naturally in the recurrence
discussion of \cite[Proposition~4.6]{GlasnerWeiss}.  The example shows why the
minimal hypothesis in Question~\ref{ques:huang-ye} is substantive.

\subsection{Local SIP-rigidity and proof of Theorem~\ref{thmB:intro}}
\label{subsec:local-sip}

Recall that, for a bounded function $h$ and a nonempty set $U\subseteq X$,
we write
\[
 \|h\|_U=\sup_{x\in U}|h(x)|.
\]

\begin{definition}
\label{def:local-sip-rigid}
Let $(r_i)_{i\geq1}\subseteq\N$ be strictly increasing.  A nonconstant
function $f\in C(X,[0,1])$ is \emph{locally SIP-rigid along
$(r_i)$} if there is a nonempty open set $U\subseteq X$ such that
\begin{equation}
 \lim_{M\to\infty}
 \sup_{n\in\SIP(r_i:i>M)}
 \|f\circ T^n-f\|_U=0.
 \label{eq:local-sip-rigid}
\end{equation}
\end{definition}

A function is called \emph{locally SIP-rigid} if it is locally
SIP-rigid along some strictly increasing sequence in $\N$ in the sense of
Definition~\ref{def:local-sip-rigid}.  The normalization $f:X\to[0,1]$
entails no loss of generality: every nonconstant real-valued continuous
function can be affinely rescaled, and the defining local norms are
multiplied by a fixed positive constant.

\begin{theorem}
\label{thm:mildmixing-sip-functions}
A topological dynamical system is mildly mixing if and only if it has no
nonconstant locally SIP-rigid continuous function.
\end{theorem}

\begin{proof}
Assume first that $(X,T)$ is mildly mixing, and suppose that $f$ is
locally SIP-rigid on a nonempty open set $U$ along $(r_i)$.  Choose
$x_0\in U$ and $y\in X$ with
\[
 \epsilon=|f(y)-f(x_0)|>0.
\]
By continuity there are nonempty open sets $U_0\subseteq U$ and $V\ni y$
such that
\[
 |f(v)-f(u)|>\epsilon/2
 \quad(u\in U_0,\ v\in V).
\]
Choose $M$ so large that
\[
 \sup_{n\in\SIP(r_i:i>M)}
 \|f\circ T^n-f\|_U<\epsilon/2.
\]
If $n\in N(U_0,V)\cap\SIP(r_i:i>M)$, choose
$z\in U_0$ with $T^nz\in V$.  The two preceding inequalities then
contradict each other.  Hence
\[
 N(U_0,V)\cap\SIP(r_i:i>M)=\emptyset,
\]
contrary to Lemma~\ref{lem:sip-multiplier}.  Thus a mildly mixing system has
no nonconstant locally SIP-rigid function.

Conversely, suppose that $(X,T)$ is not mildly mixing.  By
Lemma~\ref{lem:sip-multiplier}, there are nonempty open sets
$A,B\subseteq X$ and a strictly increasing sequence
$(r_i)\subseteq\N$ such that
\begin{equation}
 N(A,B)\cap\SIP(r_i:i>0)=\emptyset.
 \label{eq:sip-hole}
\end{equation}
Assume first that one can choose distinct points $a\in A$ and $b\in B$.
Choose disjoint nonempty open sets
\[
 a\in U\subseteq A,\qquad
 b\in O,\qquad \overline O\subseteq B.
\]
Put
\[
 \delta=\dist(\{b\},X\setminus O)>0,
 \qquad
 f(x)=\max\left\{0,1-\frac{d(x,b)}{\delta}\right\}.
\]
Then $f$ is nonconstant, continuous, and vanishes on $U$ and on
$X\setminus O$.  If $x\in U$ and
$n\in\SIP(r_i:i>0)$, equation~\eqref{eq:sip-hole} gives
$T^nx\notin B$, and hence
\[
 f(T^nx)=f(x)=0.
\]
Thus $f$ is locally SIP-rigid, in fact with exact equality on the whole
generated SIP-set.

The only case not covered above is $A=B=\{x\}$, where $x$ is isolated.
The point $x$ cannot be periodic.  Indeed, if its period were $q$, then
\[
 N(A,B)=q\N,
\]
and $q\N\in\SIP^{*}$.  To see the latter assertion, for an arbitrary
strictly increasing sequence $(s_i)\subseteq\N$, two of the $q+1$ partial sums
\[
 0,\ s_1,\ s_1+s_2,\ldots,s_1+\cdots+s_q
\]
are congruent modulo $q$; their positive difference belongs to
$q\N\cap\SIP(s_i:i>0)$.

Hence $x$ is aperiodic.  The point $T^{-1}x$ is isolated as well, so
\[
 f=\mathbf 1_{\{T^{-1}x\}}
\]
is continuous and nonconstant.  With $U=\{x\}$, one has
\[
 f(x)=f(T^nx)=0\qquad(n\geq1),
\]
because $T^nx=T^{-1}x$ would make $x$ periodic.  Thus $f$ is again
locally SIP-rigid.  This proves the converse.
\end{proof}

Theorem~\ref{thm:mildmixing-sip-functions} is
Theorem~\ref{thmB:intro}.

\begin{definition}
\label{def:sip-flat-bump}
A nonzero function $f\in C(X,[0,1])$ is an \emph{SIP-flat bump} if there
are a nonempty open set $U$ and a strictly increasing sequence
$(r_i)\subseteq\N$
such that
\[
 f|_U=0,
 \qquad
 \lim_{M\to\infty}
 \sup_{n\in\SIP(r_i:i>M)}\|f\circ T^n\|_U=0.
\]
It is \emph{exact} if the last norm is zero for every
$n\in\SIP(r_i:i>0)$.
\end{definition}

\begin{corollary}
\label{cor:sip-flat}
A topological dynamical system is mildly mixing if and only if it admits no
SIP-flat bump.  If mild mixing fails, the bump may be chosen exact and
Lipschitz for the fixed compatible metric.
\end{corollary}

\begin{proof}
Every SIP-flat bump is locally SIP-rigid on its open zero set.  The proof of
Theorem~\ref{thm:mildmixing-sip-functions} constructs an exact bump when
mild mixing fails.  In the distinct-point case, the distance bump is
Lipschitz.  In the isolated-point case, the singleton and its compact
complement have positive distance, so its indicator is Lipschitz as well.
\end{proof}

\section{Locally IP-rigid functions in minimal systems}
\label{sec:theorem-c}

This section specializes the local rigidity theory to minimal systems.  We
introduce the function, pseudometric, and IP-flat bump formulations of local
IP-rigidity and use the minimal $\operatorname{IP}^{*}$ return-time criterion
to prove Theorem~\ref{thmC:intro}.

\subsection{Definitions}
\label{subsec:local-ip-definitions}

Recall that a \emph{pseudometric} on $X$ satisfies the usual metric axioms
except that distinct points may have distance zero; it is continuous when it
is continuous as a function on $X\times X$.

\begin{definition}
\label{def:local-ip-rigid}
A nonconstant function $f\in C(X,[0,1])$ is \emph{locally IP-rigid} if
there are a nonempty open set $U\subseteq X$ and a sequence
$(p_i)_{i\geq1}\subseteq\N$ such that
\begin{equation}
 \IPlim_{\alpha\in\Fin}
 \|f\circ T^{p_\alpha}-f\|_U=0.
 \label{eq:local-ip-rigid}
\end{equation}

A nonzero continuous pseudometric $\rho$ on $X$ is
\emph{locally IP-rigid} if, for some $U$ and $(p_i)$,
\[
 \IPlim_{\alpha\in\Fin}
 \sup_{x\in U}\rho(T^{p_\alpha}x,x)=0.
\]
\end{definition}

As in Definition~\ref{def:local-sip-rigid}, affine normalization shows that
restricting the range to $[0,1]$ does not change the class of
obstructions.

\begin{definition}
\label{def:ip-flat-bump}
A nonzero function $f\in C(X,[0,1])$ is an \emph{IP-flat bump} if there
are a nonempty open set $U\subseteq X$ and a sequence
$(p_i)\subseteq\N$ such that
\[
 f|_U=0,
 \qquad
 \IPlim_{\alpha\in\Fin}
 \|f\circ T^{p_\alpha}\|_U=0.
\]
It is \emph{exact} if the last norm is zero for every
$\alpha\in\Fin$.
\end{definition}

The word ``local'' is essential: the convergence is required only on $U$,
not on all of $X$.

\subsection{Proof of Theorem~\ref{thmC:intro}}
\label{subsec:proof-theorem-c}

\begin{theorem}
\label{thm:minimal-ip-functions}
Let $(X,T)$ be minimal.  The following are equivalent:
\begin{enumerate}[label=\textup{(\roman*)}]
\item $(X,T)$ is mildly mixing;
\item there is no locally IP-rigid continuous function on $X$;
\item there is no locally IP-rigid continuous pseudometric on $X$;
\item there is no exact IP-flat bump which is Lipschitz with respect to a
      fixed compatible metric on $X$.
\end{enumerate}
\end{theorem}

\begin{proof}
We first compare functions and pseudometrics.  If $f$ is locally
IP-rigid, then
\[
 \rho_f(x,y)=|f(x)-f(y)|
\]
is a nonzero locally IP-rigid continuous pseudometric.  Conversely, suppose
$\rho$ is such a pseudometric.  Choose $a,b\in X$ with
$\rho(a,b)>0$, put
\[
 C=\max\{1,\max_{x\in X}\rho(x,a)\},
 \qquad
 f(x)=\rho(x,a)/C.
\]
Then $f$ is nonconstant and
\[
 |f(T^{p_\alpha}x)-f(x)|
 \leq C^{-1}\rho(T^{p_\alpha}x,x).
\]
Thus \textup{(ii)} and \textup{(iii)} are equivalent.

Assume \textup{(i)} and let $f$ satisfy
\eqref{eq:local-ip-rigid} on a nonempty open set $U$.  Choose
$x_0\in U$, $y\in X$, and
\[
 \epsilon=|f(y)-f(x_0)|>0.
\]
There are nonempty open sets $U_0\subseteq U$ and $V\ni y$ such that
\[
 |f(v)-f(u)|>\epsilon/2
 \quad(u\in U_0,\ v\in V).
\]
Choose $\eta\in\Fin$ so that
\[
 \|f\circ T^{p_\alpha}-f\|_U<\epsilon/2
 \quad(\alpha>\eta).
\]
It follows that
\[
 N(U_0,V)\cap
 \operatorname{FS}(p_i:i>\max\eta)=\emptyset,
\]
contradicting the $\operatorname{IP}^{*}$ criterion in
Lemma~\ref{lem:minimal-ipstar}.  Hence
\textup{(i)} implies \textup{(ii)}.

Every exact IP-flat bump is locally IP-rigid on its open zero set, so
\textup{(ii)} implies \textup{(iv)}.  It remains to prove the converse
contrapositively.  Suppose $(X,T)$ is not mildly mixing.  By
Lemma~\ref{lem:minimal-ipstar}, there are nonempty open sets
$A,B\subseteq X$ and an IP-set $F$ such that
\[
 F\cap N(A,B)=\emptyset.
\]
Choose $\operatorname{FS}(p_i)\subseteq F$.

First suppose $X$ is infinite.  An infinite minimal system has no isolated
points, so choose distinct $a\in A$ and $b\in B$.  Choose nonempty open
sets $U,W,O$ satisfying
\[
 a\in U\subseteq A,\qquad
 b\in W,\qquad
 \overline W\subseteq O\subseteq\overline O\subseteq B,
 \qquad
 U\cap\overline O=\emptyset.
\]
Put $K=\overline W$,
\[
 \delta=\dist(K,X\setminus O)>0,
 \qquad
 f(x)=\max\left\{0,1-\frac{d(x,K)}{\delta}\right\}.
\]
Then $f$ is nonzero, Lipschitz, and vanishes on $U$ and outside $O$.
For $x\in U$ and $n\in\operatorname{FS}(p_i)$, the IP-hole gives
$T^nx\notin B$, whence
\[
 f(T^nx)=f(x)=0.
\]
Thus $f$ is an exact Lipschitz IP-flat bump.

If $X$ is finite and nontrivial, minimality makes it a periodic orbit of
some period $q\geq2$.  Choose $x\in X$, put $U=\{x\}$, let
\[
 f=\mathbf 1_{\{Tx\}},
 \qquad
 p_i=q\quad(i\geq1).
\]
Every $p_\alpha$ is a multiple of $q$, so
\[
 f(x)=f(T^{p_\alpha}x)=0
 \qquad(\alpha\in\Fin).
\]
Every function on a finite metric space is Lipschitz.  Hence we again obtain
an exact Lipschitz IP-flat bump.  The one-point system is mildly mixing.
This proves the contrapositive and completes the equivalence.
\end{proof}

Theorem~\ref{thm:minimal-ip-functions} proves
Theorem~\ref{thmC:intro}.

\begin{corollary}%[Quantitative form]
\label{cor:quantitative-local-ip}
A minimal system is mildly mixing if and only if, for every nonconstant
$f\in C(X,[0,1])$, every nonempty open $U\subseteq X$, and every sequence
$(p_i)\subseteq\N$, there is $\varepsilon>0$ such that, for every
$\eta\in\Fin$, one can find $\alpha>\eta$ with
\[
 \|f\circ T^{p_\alpha}-f\|_U\geq\varepsilon.
\]
\end{corollary}

\begin{proof}
This is exactly the negation of the full-tail limit in
Definition~\ref{def:local-ip-rigid}, combined with
Theorem~\ref{thm:minimal-ip-functions}.
\end{proof}

\section{Factors determined by IP-rigid functions}
\label{sec:factors}

This section studies the factors determined by globally and locally IP-rigid
functions.  We construct the global and marked local factors, give a metric
condition under which local rigidity globalizes, analyze the local rigidity
algebra and its invariant core, and isolate the precise factor question that
remains open.

\subsection{Global IP-rigidity}
\label{subsec:global-ip-factor}

\begin{definition}
\label{def:global-ip-rigid}
Fix a sequence $\mathbf p=(p_i)_{i\geq1}\subseteq\N$.  A function
$f\in C(X)$ is \emph{globally IP-rigid along $\mathbf p$} if
\[
 \IPlim_{\alpha\in\Fin}
 \|f\circ T^{p_\alpha}-f\|_\infty=0.
\]
\end{definition}

\begin{proposition}%[Factor determined by a global IP-rigid function]
\label{prop:global-ip-factor}
Let $f\in C(X,[0,1])$.  Then $f$ is globally IP-rigid along
$\mathbf p$ if and only if
\[
 \IPlim_{\alpha\in\Fin}
 \sup_{y\in X_f}
 \rho_{\pm}(\sigma^{p_\alpha}y,y)=0.
\]
In this case the bilateral orbit-name factor
$(X_f,\sigma)$ is uniformly rigid.  It is the smallest factor
through which $f$ factors.
\end{proposition}

\begin{proof}
For every $j\in\Z$, surjectivity of $T^j$ gives
\[
 \|f\circ T^{j+p_\alpha}-f\circ T^j\|_\infty
 =\|f\circ T^{p_\alpha}-f\|_\infty.
\]
The metric $\rho_{\pm}$ therefore yields
\[
 \sup_{y\in X_f}
 \rho_{\pm}(\sigma^{p_\alpha}y,y)
 \leq
 \|f\circ T^{p_\alpha}-f\|_\infty.
\]
The reverse implication follows from the zeroth coordinate, whose weight is
$1/3$.

Choose successive finite blocks
$\beta_1<\beta_2<\cdots$ such that
\[
 r_k=p_{\beta_k}
\]
is strictly increasing.  This is possible because $p_i\geq1$: each new
block may be taken farther out and long enough to have a larger sum.  Since
$\beta_k$ eventually lies beyond every prescribed tail index, the
full-tail IP-convergence gives $\sigma^{r_k}\to\id$ uniformly, so the
factor is uniformly rigid.
Minimality of this factor among factors carrying $f$ was proved in
Proposition~\ref{prop:bilateral-rigid-factor}(iii).
\end{proof}

\subsection{The canonical factor with marked local data}
\label{subsec:marked-local-factor}

Local IP-rigidity does not initially give uniform convergence in a factor
metric.  The factor must retain the observable and its local data.

\begin{definition}%[Marked locally IP-rigid factor]
\label{def:marked-local-factor}
A \emph{marked locally IP-rigid factor} of $(X,T)$ consists of a nontrivial
factor map
\[
 q:(X,T)\longrightarrow(Y,S),
\]
a nonconstant function $g\in C(Y,[0,1])$, a nonempty open set
$W\subseteq Y$, and a sequence $(p_i)\subseteq\N$ such that
\[
 \IPlim_{\alpha\in\Fin}
 \|g\circ S^{p_\alpha}-g\|_W=0.
\]
The observable, open set, and sequence are part of the marked data.
\end{definition}

\begin{theorem}%[Canonical factor with marked local data]
\label{thm:canonical-marked-factor}
Let $(X,T)$ be minimal, and let the nonconstant function
$f\in C(X,[0,1])$ be locally IP-rigid on a nonempty open set $U$ along
$\mathbf p=(p_i)$.  Then the canonical bilateral orbit-name factor has the
following properties:
\begin{enumerate}[label=\textup{(\roman*)}]
\item $(X_f,\sigma)$ is a nontrivial minimal factor of $(X,T)$;
\item for the nonconstant zeroth-coordinate function $\widehat f(y)=y_0$,
      there is a nonempty open $W\subseteq X_f$ such that
      \[
       \IPlim_{\alpha\in\Fin}
       \|\widehat f\circ\sigma^{p_\alpha}-\widehat f\|_W=0;
      \]
\item $X_f$ is the smallest factor through which $f$ factors.
\end{enumerate}
Consequently, a minimal system is mildly mixing if and only if it has no
nontrivial marked locally IP-rigid factor.
\end{theorem}

\begin{proof}
Proposition~\ref{prop:bilateral-rigid-factor}(i) shows that
$\Phi_f:X\to X_f$ is a factor map.  Its image is nontrivial
because $f$ is nonconstant, and a factor of a minimal system is minimal.

Choose a nonempty open set $V$ with $\overline V\subseteq U$.
Minimality and compactness give $a_1,\ldots,a_s\in\Nzero$ such that
\[
 X=\bigcup_{\ell=1}^{s}T^{-a_\ell}V.
\]
Therefore
\[
 X_f
 =\bigcup_{\ell=1}^{s}
 \sigma^{-a_\ell}\Phi_f(\overline V).
\]
This is a finite cover by closed sets.  By the Baire category theorem, one
member has nonempty interior.  Since $\sigma$ is a homeomorphism,
\[
 W=\Int_{X_f}\Phi_f(\overline V)
\]
is nonempty.  If $y\in W$, choose $x\in\overline V$ with
$\Phi_f(x)=y$.  Then
\[
\begin{split}
 |\widehat f(\sigma^{p_\alpha}y)-\widehat f(y)|
 &=|f(T^{p_\alpha}x)-f(x)|\\
 &\leq\|f\circ T^{p_\alpha}-f\|_U.
\end{split}
\]
This proves \textup{(ii)}.  Assertion \textup{(iii)} is
Proposition~\ref{prop:bilateral-rigid-factor}(iii).

Finally, suppose that $(q,g,W,\mathbf p)$ is marked local data, and put
$f=g\circ q$ and $U=q^{-1}(W)$.  Since $q(U)=W$ and $qT^n=S^nq$, one has
\[
 \|f\circ T^{p_\alpha}-f\|_U
 =\|g\circ S^{p_\alpha}-g\|_W.
\]
Thus a marked observable pulls back to a nonconstant locally IP-rigid
function on $X$, while the preceding construction produces a marked factor
from every locally IP-rigid function.  The last assertion now follows from
Theorem~\ref{thm:minimal-ip-functions}.
\end{proof}

\begin{remark}
\label{rem:marked-not-uniform}
Theorem~\ref{thm:canonical-marked-factor} does not say that
$(X_f,\sigma)$ is uniformly rigid.  It controls one coordinate on
one open set.  Uniform rigidity would require simultaneous control of a
point-separating family, or equivalently control of the factor metric.
\end{remark}

\subsection{When local metric rigidity globalizes}
\label{subsec:metric-globalization}

\begin{proposition}%[Metric local rigidity globalizes]
\label{prop:metric-globalization}
Let $(Y,S)$ be a minimal topological dynamical system with compatible
metric $d_Y$.  Suppose that, for a nonempty open $W\subseteq Y$ and a
sequence $(p_i)\subseteq\N$,
\[
 \IPlim_{\alpha\in\Fin}
 \sup_{y\in W}d_Y(S^{p_\alpha}y,y)=0.
\]
Then
\[
 \IPlim_{\alpha\in\Fin}
 \sup_{y\in Y}d_Y(S^{p_\alpha}y,y)=0.
\]
In particular, $(Y,S)$ is uniformly rigid.
\end{proposition}

\begin{proof}
Choose $a_1,\ldots,a_s\in\Nzero$ such that
\[
 Y=\bigcup_{\ell=1}^{s}S^{-a_\ell}W.
\]
Given $\varepsilon>0$, uniform continuity of the finitely many maps
$S^{-a_\ell}$ gives $\delta>0$ such that
\[
 d_Y(v,w)<\delta
 \quad\Longrightarrow\quad
 d_Y(S^{-a_\ell}v,S^{-a_\ell}w)<\varepsilon
\]
for every $\ell$.  For every sufficiently late $\alpha$, the local
supremum is less than $\delta$.  Given $y\in Y$, choose $\ell$ with
$u=S^{a_\ell}y\in W$.  Commutativity of the powers of $S$ gives
\[
\begin{split}
 d_Y(S^{p_\alpha}y,y)
 &=d_Y(S^{-a_\ell}S^{p_\alpha}u,S^{-a_\ell}u)
 <\varepsilon.
\end{split}
\]
This proves global IP-convergence.  Extracting strictly increasing block
sums $r_k=p_{\beta_k}$ gives
$S^{r_k}\to\id_Y$ uniformly.
\end{proof}

\subsection{The local rigidity algebra and its invariant core}
\label{subsec:local-core}

\begin{definition}
\label{def:local-core}
For a nonempty open $U\subseteq X$ and a sequence
$\mathbf p=(p_i)\subseteq\N$, define
\begin{align}
 \mathcal L(U,\mathbf p)
 &=\left\{f\in C(X):
   \IPlim_{\alpha\in\Fin}
   \|f\circ T^{p_\alpha}-f\|_U=0\right\},
 \label{eq:local-algebra}\\
 \mathcal C(U,\mathbf p)
 &=\left\{f\in C(X):
   f\circ T^j\in\mathcal L(U,\mathbf p)
   \text{ for every }j\in\Z\right\}.
 \label{eq:invariant-core}
\end{align}
We call $\mathcal C(U,\mathbf p)$ the
\emph{$T^{\pm1}$-invariant core} of the local rigidity algebra.
Here and below, $\mathbb R\mathbf1$ denotes the algebra of constant
real-valued functions.
\end{definition}

\begin{theorem}%[Factor determined by the invariant core]
\label{thm:core-factor}
Let $(X,T)$ be minimal.  For every pair $(U,\mathbf p)$:
\begin{enumerate}[label=\textup{(\roman*)}]
\item $\mathcal L(U,\mathbf p)$ is a closed unital subalgebra of $C(X)$,
      and $\mathcal C(U,\mathbf p)$ is its largest closed unital
      $T^{\pm1}$-invariant subalgebra;
\item $\mathcal C(U,\mathbf p)$ determines a canonical factor
      \[
       q^{\mathrm{core}}_{U,\mathbf p}:(X,T)\longrightarrow
       (Y^{\mathrm{core}}_{U,\mathbf p},
        S^{\mathrm{core}}_{U,\mathbf p})
      \]
      which is maximal among factors $q:(X,T)\to(Y,S)$ satisfying
      $q^{*}C(Y)\subseteq\mathcal L(U,\mathbf p)$;
\item if $\mathcal C(U,\mathbf p)$ contains a nonconstant function, then
      its factor is nontrivial and uniformly rigid.
\end{enumerate}
\end{theorem}

\begin{proof}
The triangle inequality shows that $\mathcal L(U,\mathbf p)$ is a vector
space containing the constants.  If $f,g\in\mathcal L(U,\mathbf p)$, then
\[
\begin{split}
 \|(fg)\circ T^{p_\alpha}-fg\|_U
 &\leq
 \|f\|_\infty\|g\circ T^{p_\alpha}-g\|_U\\
 &\quad+
 \|g\|_\infty\|f\circ T^{p_\alpha}-f\|_U,
\end{split}
\]
so it is an algebra.  If $f_n\in\mathcal L(U,\mathbf p)$ and $f_n\to f$
uniformly, then the estimate used in the proof of
Theorem~\ref{thm:maximal-rigid-factor}, namely
\[
 \|f\circ T^{p_\alpha}-f\|_U
 \leq2\|f-f_n\|_\infty
 +\|f_n\circ T^{p_\alpha}-f_n\|_U,
\]
shows that it is closed under uniform limits.
Definition~\eqref{eq:invariant-core} shows that
$\mathcal C(U,\mathbf p)$ is closed, unital, and invariant under both
$U_T$ and $U_T^{-1}$.  Every $T^{\pm1}$-invariant subalgebra contained
in $\mathcal L(U,\mathbf p)$ is contained in the core.  This proves
\textup{(i)}.

Apply Lemma~\ref{lem:factor-algebra} to
$\mathcal C(U,\mathbf p)$.  If a factor $q:X\to Y$ satisfies
\[
 q^{*}C(Y)\subseteq\mathcal L(U,\mathbf p),
\]
then its pullback algebra is $T^{\pm1}$-invariant and hence is contained
in $\mathcal C(U,\mathbf p)$.  The universal part of
Lemma~\ref{lem:factor-algebra} makes $q$ factor through the core factor.
This proves \textup{(ii)}.

Assume now that the core is nonconstant.  Choose a nonempty open set $V$
with $\overline V\subseteq U$.  By minimality, there are
$a_1,\ldots,a_s\in\Nzero$ such that
\[
 X=\bigcup_{\ell=1}^{s}T^{-a_\ell}V.
\]
Writing $q=q^{\mathrm{core}}_{U,\mathbf p}$ and
$S=S^{\mathrm{core}}_{U,\mathbf p}$, we obtain the finite closed cover
\[
 Y^{\mathrm{core}}_{U,\mathbf p}
 =\bigcup_{\ell=1}^{s}
 S^{-a_\ell}q(\overline V).
\]
Baire category gives nonempty interior to one member of this cover.  Since
each factor iterate is a homeomorphism, it follows that
\[
 W=\Int q(\overline V)\neq\emptyset.
\]
Choose a sequence $(f_m)$ dense in the closed unit ball of the core,
identify the core factor with the coordinate realization furnished by
Lemma~\ref{lem:factor-algebra}, and equip it with
\[
 d_Y(y,z)=\sum_{m=1}^{\infty}2^{-m}|y_m-z_m|.
\]
For $M\geq1$ and $\alpha\in\Fin$,
\[
 \sup_{y\in W}d_Y(S^{p_\alpha}y,y)
 \leq
 \sum_{m=1}^{M}2^{-m}
 \|f_m\circ T^{p_\alpha}-f_m\|_U+2\sum_{m>M}2^{-m}.
\]
First choose $M$ so that the second term is small, and then choose one
common IP-tail on which the finitely many local norms in the first term are
small.  Hence
\[
 \IPlim_{\alpha\in\Fin}
 \sup_{y\in W}d_Y(S^{p_\alpha}y,y)=0.
\]
The core factor is minimal because it is a factor of $(X,T)$.
Proposition~\ref{prop:metric-globalization} now gives uniform rigidity of the
core factor.
\end{proof}

\begin{remark}[Two extremal cases]
\label{rem:extremal-cores}
The invariant-core picture is transparent at two extremes.  If $(X,T)$
is minimal and mildly mixing, Theorem~\ref{thm:minimal-ip-functions} gives
\[
 \mathcal L(U,\mathbf p)=\mathcal C(U,\mathbf p)=\mathbb R\mathbf1
\]
for every nonempty open $U$ and every sequence
$\mathbf p=(p_i)\subseteq\N$.
At the opposite extreme, let $T=R_\alpha$ be an irrational rotation of
$\Torus$.  Choose a strictly increasing sequence $(p_i)\subseteq\N$ such that
$\|p_i\alpha\|_{\Torus}<2^{-i}$.  For every finite $\beta$ with
$\min\beta>N$,
\[
 \|p_\beta\alpha\|_{\Torus}
 \leq\sum_{i\in\beta}2^{-i}
 \leq\sum_{i>N}2^{-i}.
\]
Uniform continuity therefore makes every member of $C(\Torus)$ globally
IP-rigid along $\mathbf p$, and hence
\[
 \mathcal L(U,\mathbf p)=\mathcal C(U,\mathbf p)=C(\Torus)
\]
for every nonempty open $U$.  The unresolved local-to-global issue lies
between these two cases: a local algebra may be nontrivial while it is not
known, in general, whether some invariant core must be nontrivial.
\end{remark}

\begin{remark}[Why minimality is necessary]
\label{rem:minimality-core}
The globalization statement fails without minimality.  Let
\[
 Y=\Torus\sqcup\{0,1\}^{\Z},
 \qquad
 S=R_\alpha\sqcup\sigma,
\]
where $R_\alpha$ is an irrational rotation and $\sigma$ is the two-sided
full shift.  Choose a strictly increasing sequence $(p_i)\subseteq\N$ so that
\[
 \|p_i\alpha\|_{\Torus}<2^{-i},
\]
and let $U=\Torus$, which is clopen in $Y$.  Every continuous function
and every one of its $S$-translates is locally IP-rigid on $U$.
Indeed, if $\min\beta>N$, then
\[
 \|p_\beta\alpha\|_{\Torus}
 \leq\sum_{i\in\beta}2^{-i}
 \leq\sum_{i>N}2^{-i},
\]
and uniform continuity of the restriction of the function (or any fixed
translate) to the circle gives the required local IP-limit.  Therefore
\[
 \mathcal C(U,\mathbf p)=C(Y).
\]
Nevertheless $Y$ is not uniformly rigid because its full-shift component
is not uniformly rigid.  Thus minimality is essential in
Proposition~\ref{prop:metric-globalization} and
Theorem~\ref{thm:core-factor}.
\end{remark}

\subsection{The exact form of the uniformly rigid factor question}
\label{subsec:exact-open-question}

\begin{corollary}
\label{cor:exact-core-form}
For minimal topological dynamical systems,
\begin{align*}
 (X,T)\text{ is not mildly mixing}
 &\Longleftrightarrow
 \mathcal L(U,\mathbf p)\neq\mathbb R\mathbf1
 \text{ for some }(U,\mathbf p),\\
 (X,T)\text{ has a nontrivial uniformly rigid factor}
 &\Longleftrightarrow
 \mathcal C(U,\mathbf p)\neq\mathbb R\mathbf1
 \text{ for some }(U,\mathbf p).
\end{align*}
\end{corollary}

\begin{proof}
The first equivalence is Theorem~\ref{thm:minimal-ip-functions}, together
with affine normalization of a nonconstant real-valued function.  One
direction of the second equivalence is
Theorem~\ref{thm:core-factor}(iii).

Conversely, suppose
$q:(X,T)\to(Y,S)$ is a nontrivial uniformly rigid factor.  Pass to a
subsequence $(p_i)$ of a uniform rigidity sequence so that
\[
 \sup_{y\in Y}d_Y(S^{p_i}y,y)<2^{-i}.
\]
By \eqref{eq:uniform-subadditivity}, applied to $(Y,S)$, for every finite
$\alpha$,
\[
 \sup_{y\in Y}d_Y(S^{p_\alpha}y,y)
 \leq\sum_{i\in\alpha}2^{-i}.
\]
Uniform continuity therefore shows that every member of $q^{*}C(Y)$ is
globally IP-rigid along $\mathbf p$.  Hence
\[
 q^{*}C(Y)
 \subseteq\mathcal L(X,\mathbf p)
 =\mathcal C(X,\mathbf p).
\]
Here the equality follows because surjectivity of every $T^j$, $j\in\Z$,
makes the global norm invariant under composition with $T^j$; hence
$\mathcal L(X,\mathbf p)$ is already $T^{\pm1}$-invariant.
The right-hand core is nonconstant because the factor is nontrivial.
\end{proof}

Corollary~\ref{cor:exact-core-form} isolates the unresolved implication in
Question~\ref{ques:huang-ye}.  It asks whether nontriviality of some local
rigidity algebra must force nontriviality of an invariant core, possibly for
different local data.  The hierarchy proved in this section is:
\[
\begin{array}{c}
 \text{global IP-rigid observable}\\
 \Downarrow\\
 \text{uniformly rigid bilateral orbit-name factor},
\end{array}
\qquad
\begin{array}{c}
 \text{locally IP-rigid observable}\\
 \Downarrow\\
 \text{marked local factor},
\end{array}
\]
while a uniformly rigid ordinary factor follows from the local data precisely
when enough orbit translates survive in the invariant core.  No resolution of
Question~\ref{ques:huang-ye} is asserted here.

% -------------------------------------------------------------------------
% AUTHOR ACTION REQUIRED BEFORE SUBMISSION/PROOF APPROVAL:
% Insert the complete Funding statement here.  The source supplied for this
% revision contains no funding information, so none has been invented.
% Example syntax:
% \subsection*{Funding}
% This work was supported by ... [grant number ...].
% -------------------------------------------------------------------------

\end{document}